\documentclass[11pt]{amsart}
\usepackage{amsmath, amssymb}
\usepackage{eucal}

\numberwithin{equation}{section}

\newtheorem{thm}{Theorem}[section]

\newtheorem{cor}[thm]{Corollary}

\theoremstyle{remark}

\theoremstyle{definition}

{\end{list}}

\newcommand{\pd}{\partial}

\newcommand{\s}{\sigma}

\newcommand{\sk}[2]{\langle #1 ,\, #2 \rangle}
\newcommand{\Sk}[2]{\left \langle #1 ,\, #2 \right \rangle}
\newcommand{\nor}[1]{\| #1 \|}

\newcommand{\NOR}[1]{\Big |\hskip -0.7pt\Big | #1 \Big |\hskip
-0.7pt\Big |}

\newcommand{\C}{\mathbb{C}}
\newcommand{\R}{\mathbb{R}}

\newcommand{\cE}{\mathcal{E}}
\newcommand{\cV}{\mathcal{V}}

\newcommand\spn{\operatorname{span}}

\begin{document}

\title{A theorem about three quadratic forms}

\author{Oliver Dragi\v cevi\'c, Sergei Treil, Alexander Volberg}
\address{Institute of Mathematics, Physics and Mechanics, University of Ljubljana, Slovenia}
\email{oliver.dragicevic@fmf.uni-lj.si}

\address{Department of Mathematics, Brown University, 151 Thayer
Str., Box 1917,
 Providence, RI  02912, USA }
\email{treil@math.brown.edu}
\urladdr{http://www.math.brown.edu/\~{}treil}

\address{Department of Mathematics, Michigan State University, East Lansing, MI. 48824, USA; and School of Mathematics, University of Edinburgh, Edinburgh EH9 3JZ, UK }
\email{volberg@math.msu.edu}
\urladdr{http://www.math.msu.edu/\~{}volberg}

\date{March, 2007}

%

\maketitle

\setcounter{tocdepth}{1}



\section{A funny theorem}
\begin{thm}
\label{tq1.1}
Let $\s_0, \s_1, \s_2$ be non-negative non-zero quadratic forms on a vector  space $\cV$ (real or complex) satisfying
\begin{equation}
\label{q1.1}
\s_0[x] \ge 2\sqrt{\s_1[x]\cdot \s_2[x]}, \qquad \forall x 
\in \cV.
\end{equation}
Then there exists a constant $\alpha >0$ such that 
\begin{equation}
\label{q1.2}
\s_0[x] \ge \alpha \s_1[x] + \frac1{\alpha} \s_2[x], \  \forall x 
\in \cV .
\end{equation}
\end{thm}

The condition \eqref{q1.2} clearly implies \eqref{q1.1}. 

\begin{proof}[Proof of Theorem \ref{tq1.1}]
Consider the family of quadratic forms $\s^s$,  $s\in (0, \infty)$
$$
\s^s := \s_0 - s \s_1 -s^{-1} \s_2. 
$$
Assume that the theorem is not true, i.e.~that for all $s$ the form  $\s^s$ is not
non-negative to get a contradiction.

First of all let us notice, that if for some $x\in \cV$ we have $\s_2[x] /\s_1[x] = s^2$ then 
$$
s \s_1[x] + s^{-1} \s_2[x] = 2\sqrt{\s_1[x]\cdot \s_2[x]}. 
$$
Therefore, if we find $s>0$ such that 
$$
\s^s[x] <0
$$
for some $x$ satisfying $\s_2[x] /\s_1[x] = s^2$, we get a contradiction with
\eqref{q1.1}.

Consider the set $S\subset (0, \infty)\times (0, \infty)$ consisting of all pairs
$(s, t)$ such that $\s^s[x]<0$ for some $x$ satisfying $\s_2[x] /\s_1[x] = t^2$. 
We want to show that $(\alpha, \alpha)\in S$ for some $\alpha>0$, which gives us
the contradiction.

Since all forms are non-zero, there exist vectors $x_k$, $k=1, 2$ such that $\s_k[x_k]>0$, $k=1, 2$. Therefore there exists a linear combination $x=\alpha x_1 + \beta x_2$ such that $\s_1[x], \s_2[x]>0$ (to see that one only needs to consider forms on two-dimensional space $\mathcal L\{x_1, x_2\}$). 
One
concludes that $\s^s[x]<0$ for all sufficiently large and for all sufficiently small $s$. So, if $t_0:= \sqrt{\s_2[x]/\s_1[x]}$, the points $(s, t_0)$ belong to $S$ for all sufficiently small and for all sufficiently large $s$. Thus $S$ has points on both sides of the line $s=t_0$. So, if we prove that the set $S$ is connected, it must contain a point $(\alpha, \alpha)$, which gives us the desired contradiction.

We now prove the following properties of the set $S$:

\begin{enumerate}
\item For any $(s_0, t_0)\in S$ we have $(s, t_0)\in S$ for all $s$ in a small neighborhood of $s_0$;
\item Projection of $S$ onto the $s$-axis is the whole ray $(0, \infty)$;
\item For any $s\in (0, \infty)$ the set $\{t : (s, t)\in S\}$ is an  interval;
\item $S$ is connected. 
\end{enumerate}
The property 1 follows immediately from the continuity of the function $s\mapsto \s^s[x]$ ($x$ is fixed). 

The property 2 is just our assumption that $\s^s$ is never positive semi-definite. 

Property 3 requires some work. Let $s$ be fixed. Suppose that $(s,t_k)\in S$, $k=1, 2$, i.e.~that there exist vectors $x_{1}, x_2\in \cV$ such that $\s^s [x_k]<0$ and $t_k =\sqrt{\s_2[x_k]/\s_1[x_k]}$, $k=1, 2$. Consider the (real) subspace $\cE\subset \cV$, $\cE=\spn_{\text{real}}\{x_1, x_2\}$, and let us restrict all quadratic forms onto $\cE$.  

For a vector $x\in \cE$ satisfying
$\s^s[x]<0$ define 
$$
\tau(x) := \sqrt{\s_2[x]/\s_1[x]}. 
$$
Notice, that by the definition of $\s^s$  for any  $x$ satisfying $\s^s[x]<0$ both $\s_1[x]$ and $\s_2[x]$
 cannot be simultaneously $0$, so $\tau: \{x\in \cE: \s^s[x]<0\} \to [0, \infty]$ is a well defined continuous map (we are allowing $\tau(x)=+\infty$).

Since $\s^s[x_k]<0$ the quadratic form $\s^s\bigm| \cE$ has either one or two negative squares. In the latter case the set $K=\{x\in \cE: \s^s[x]<0\}$ is the whole plane without the origin, so it is connected. In the former case it consists of two connected parts $K=K_{1}\cup K_2$, $K_1=-K_2$.  In both cases the set $\tau(K)$ is connected. Indeed, if $K$ is connected, $\tau(K)$ is  a continuous image of a connected set.  In the second case, $\tau(K_1) $ is connected and since $\tau(x)=\tau(-x)$ we have $\tau(K_1) =\tau(K_2) =\tau(K)$. So the set $\tau(K)$ contains the whole interval between the points $t_1$ and $t_2$. 

Since $\tau(K)\cap (0,\infty)\subset\{t: (s, t)\in S\}$ we can conclude that for arbitrary $t_1, t_2\in \{t: (s, t)\in S\}$, $t_1<t_2$, the whole interval $[t_1, t_2]$ belongs to the set. But that exactly means that the set $\{t: (s, t)\in S\}$ is an interval.

And now let us prove property 4 (and so the theorem). Suppose we split $S$ into 2
nonempty disjoint relatively open subsets $S=S_1\cup S_2$. Let $P$ denote the coordinate projection onto the
$s$-axis. Property 1 implies that  the sets $PS_1$, $PS_2$ are open. Property 2
implies that $PS_1\cup PS_2 =(0, \infty)$ so it follows from the connectedness
of $(0, \infty)$ that $PS_1\cap PS_2 \ne \varnothing$.

Therefore for some $s$ there exist $t_1, t_2$ such that $(s, t_k)\in S_k$, $k=1,
2$. By property 3 the whole interval $J = \{ (s, \theta t_1 + (1-\theta)t_2):
\theta \in [0,1]\}$ belongs to $S$. Therefore $J$ can be represented as a union
$J= (J\cap S_1)\cup (J\cap S_2)$ of disjoint nonempty relatively open subsets, which is impossible.

\end{proof}

\section{Consequences. The claim that goes ``against" geometric average of arithmetic average versus arithmetic average of geometric average inequality.}

We list several obvious corollaries of the previous result. Namely, Theorem \ref{lelo} follows immediately from the previous result, Theorem \ref{lelo1} is an obvious consequence of  Theorem \ref{lelo}, and Theorem \ref{lelo2} is a particular case of Theorem \ref{lelo1}. 
Theorem \ref{lelo2} was proved in \cite{DV1}, \cite{DV} (and preprints preceding  these papers). The proof was basically geometric. Actually the same proof can give more general Theorem \ref{lelo}, and even our main result. But the proof of the main result given above is easier, more concise, and is better suited for the general situation.

Then we explain how  Theorem \ref{lelo2} finds the application for the dimension free estimates of singular operators (Riesz transforms of various kind).

\begin{thm}
\label{lelo}
Suppose ${\mathcal H}$ is a (real or complex) Hilbert
space, $A,B$ are two positive-definite operators on ${\mathcal H}$. Let $T$ be a self-adjoint operator on ${\mathcal H}$ such that 
\begin{equation}
\label{Joe}
\langle Th, h\rangle\geqslant 2(\sk{Ah}{h}\sk{Bh}{h})^{\frac12}
\end{equation}
for all $h\in {\mathcal H}$. Then there exists $\tau >0$, satisfying
\begin{equation*}
\sk{Th}{h}\geqslant \tau \sk{Ah}{h}+\tau^{-1}
\sk{Bh}{h}\,
\end{equation*}
again for all $h\in {\mathcal H}$. 
\end{thm}

\begin{thm}
\label{lelo1}
Suppose ${\mathcal H}$ is a (real or complex) Hilbert
space, $A,B$ are two positive-definite operators on ${\mathcal H}$. Let $T$ be a self-adjoint operator on ${\mathcal H}$ such that 
\begin{equation}
\label{Joe1}
\langle Th, h\rangle\geqslant 2(\sk{Ah}{h}\sk{Bh}{h})^{\frac12}
\end{equation}
for all $h\in {\mathcal H}$. Then
\begin{equation*}
\text{trace}\, T\geqslant 2 (\text{trace} \,A\,\text{trace}\, B)^{\frac12}\,.
\end{equation*}
 \end{thm}
 
 \noindent{\bf Remark.} If we have two sequences of positive numbers
 $\{a_i\}_{i=1}^n$, $\{b_i\}_{i=1}^n$, then 
 $$\frac {(\sum a_i\sum b_i)^{\frac12}}{n}\geq\frac{\sum (a_i b_i)^{\frac12}}{n}\,.$$ 
 This is just to say that the geometric average of arithmetic averages is always greater than the arithmetic average of geometric averages. Now let us have operators $T,A,B$ as above, and any system $\{e_i\}_{i=1}^n$ of vectors in ${\mathcal H}$.
 Put $a_i = \langle Ae_i, e_i\rangle,\, b_i = \langle Be_i, e_i\rangle$ and $t_i = \langle Te_i, e_i\rangle$/2. Denote  
$A(G)=\frac{\sum (a_i b_i)^{\frac12}}{n}$ and $G(A)=\frac {(\sum a_i\sum b_i)^{\frac12}}{n}.$ We know that $G(A)\geq A(G)$ always. 
 We are given that
 $t_i \geq (a_i b_i)^{\frac12}$. So of course $\frac{\sum t_i}{n} \geq A(G)$. But for this {\it special collections of numbers} we have more:
 $$
 \frac{\sum t_i}{n} \geq G(A)\,.
 $$

\begin{thm}
\label{lelo2}
Suppose ${\mathcal H}$ is a (real or complex) finite-dimensional Hilbert
space, ${\mathcal H}_j$, $j=1,2$, are two nontrivial mutually orthogonal subspaces of\, ${\mathcal H}$ and $P_j$ are the corresponding orthogonal projections. Let $T$ be a self-adjoint operator on ${\mathcal H}$ such that 
\begin{equation}
\label{Joe2}
\sk{Th}{h}\geqslant 2\nor{P_1h}\nor{P_2h}
\end{equation}
for all $h\in {\mathcal H}$. Let also $L$ be a Hilbert-Schmidt operator acting from any Hilbert space into ${\mathcal H}$. Under assumption \eqref{Joe2} we get then
\begin{equation}
\label{tobeused}
\text{trace}\,L^*TL \geq 2\|P_1L\|_{HS}\|P_2L\|_{HS}\,,
\end{equation}
where $\|\cdot\|_{HS}$ means the Hilbert-Schmidt norm.
\end{thm}

\begin{proof}
We apply Theorem \ref{lelo1} to $T$ which is $L^*TL$, and $A:=L^*P_1L$, $B:= L^*P_2L$.
\end{proof}

\vspace{.1in}

Now let us explain how our three-quadratic-form lemma (especially in the form of Theorem \ref{lelo2}) was used in \cite{DV}--\cite{DV4} and \cite{P}.

Let $f,g$ be two compactly supported smooth functions on $\R^n$, and ${\mathcal L}$ is a certain Laplacian (it does not matter what this means). We use Poisson extension $P_t f := e^{-t{\mathcal L}}f$. Consider $v(x,t):= (P_tf, P_tg, P_t|f|^p, P_t|g|^q)$. 
In \cite{DV} a function $B$ on $\R^4$  was constructed which has a following property ($x_{0}$ means variable $t$)
$$
\sum_{i=0}^{n} \Sk{-d^2B(v)
\frac{\pd v}{\pd x_i}}{\frac{\pd v}{\pd x_i}}\geqslant
2\sum_{i=0}^{n} \left |\frac{\pd}{\pd x_i} P_tg(x)\right
| \left |\frac{\pd}{\pd x_i} P_tf(x)\right|\,.
$$
However, this inequality can be strengthened, i.e. it actually self-improves, thanks to Theorem \ref{lelo2}.

Let $Jf, Jg$ denote Jacobians of $f,g$ (all derivatives of $P_tf, P_tg$ with respect to $x_i, i=0,1,...,n$, with $x_0$ being $t$).
By applying Theorem \ref{lelo2} to $T=-d^2B(v)$ and $L=\nabla v(x,t)$, the latter understood as an operator $\R^{n+1}\rightarrow\R^4$, 
we find  that we even have
$$
\sum_{i=1}^{n+1} \Sk{-d^2B(v)
\frac{\pd v}{\pd x_i}(x,t)}{\frac{\pd v}{\pd x_i}(x,t)} 
\geqslant
2\nor{J f(x,t)}_{HS}\nor{J g(x,t)}_{HS}\,.
$$
This inequality immediately gives the following (see \cite{DV})
$$
\int_0^\infty\int_{\R^{n}}\nor{Jf(x,t)}_{HS}\nor{Jg(x,t)}_{HS}\,t\,dx\,dt
\leqslant C(p)\nor{f}_p\nor{g}_q\,.
$$
As a quick consequence one obtains dimension free estimates of Riesz transforms associated with the chosen Laplacian (see \cite{DV}):
\begin{cor}
\label{glavni}
For every $n\in {\mathbb N}$ and every $f\in L^p$, 
$$
\NOR{\Big(\sum_{i=1}^n|R_if|^2\Big)^{1/2}}_p \leqslant C(p)\nor{f}_p
$$
and therefore also $\nor{R_k}_{B(L^p)}\leqslant C(p)$ for
$k=1,\hdots,n$.
\end{cor}

\noindent{\bf Acknowledgments}. The third author is grateful to Stanislaw Szarek for his calculations that supported our belief in truthfulness of the statement of Theorem \ref{lelo}.

\def\cprime{$'$}
\providecommand{\bysame}{\leavevmode\hbox to3em{\hrulefill}\thinspace}


\end{document}